\documentclass[11pt,reno]{amsart} %will change this back to 12pt later probably
\usepackage{amssymb}
\usepackage{amsmath}
\usepackage{mathtools}
\usepackage{enumitem}
\usepackage{mathrsfs}
\usepackage[english]{babel}
\usepackage[all]{xy}
\usepackage{hyperref}
\usepackage{marginnote}
\usepackage{comment}
\usepackage{tikz-cd}

\usepackage{stmaryrd}

\usepackage{fullpage}

\RequirePackage{mathrsfs} %\let\mathcal\mathcal

\newtheorem{theorem}{Theorem}[section]
\newtheorem{lemma}[theorem]{Lemma}

\newtheorem{corollary}[theorem]{Corollary}
\newtheorem{conjecture}[theorem]{Conjecture}
\newtheorem{alphtheorem}{Theorem}

\theoremstyle{definition}
\newtheorem*{ack}{Acknowledgments}
\newtheorem*{con}{Conventions}
\newtheorem{remark}[theorem]{Remark}
\newtheorem{example}[theorem]{Example}
\newtheorem{definition}[theorem]{Definition}

\numberwithin{equation}{section} \numberwithin{figure}{section}

\DeclareMathOperator{\Spec}{Spec}

\usepackage{color}

\definecolor{orange}{rgb}{1,0.5,0}

\title{Hilbert irreducibility   for abelian varieties over function fields of characteristic zero}

\author{Ariyan Javanpeykar}
\address{Ariyan Javanpeykar \\ 
IMAPP Radboud University Nijmegen \\
PO Box 9010, 6500GL\\
 Nijmegen, The Netherlands}
\email{ariyan.javanpeykar@ru.nl }

 %hyperbolicity


\begin{document}

\begin{abstract} 
 We prove  Hilbert's irreducibility theorem for abelian varieties over function fields of characteristic zero.
\end{abstract}

\maketitle
%\tableofcontents
\thispagestyle{empty}

\thispagestyle{empty}

\section{Introduction}
 
We prove a variant of Hilbert's irreducibility theorem for abelian varieties over function fields of characteristic zero. To state our result,  we follow Corvaja--Zannier \cite{CZHP} and consider the weak Hilbert property. 

\begin{definition}
Let $K$ be a field of characteristic zero. 
Let $X$ be a normal variety over $K$. 
\begin{enumerate}
\item A subset $\Sigma\subset X(K)$ is \emph{thin} if there is an integer $n\geq 1$ and a collection of     finite surjective morphisms $(\pi_i\colon Y_i\to X)_{i=1,\ldots,n}$ with $\deg \pi_i >1$ for all $i=1,\ldots, n$    such that $ \Sigma \setminus \cup_{i=1}^n \pi_i(Y_i(K))$ is not dense in $X$.
\item We say that $X$ has the \emph{Hilbert property over $K$}  if $X(K)$ is not thin.  
\item A field $K$ of characteristic zero is \emph{Hilbertian}  if $\mathbb{P}^1_K$ has the Hilbert property over $K$.
\item  A finite surjective morphism $Y\to X$ of integral normal noetherian schemes is a \emph{ramified cover} if it is not unramified.
\item    A subset $\Sigma$ of $X(K)$ is \emph{strongly thin (in $X$)} if there is an integer $n\geq 1$ and a collection of   ramified covers $(\pi_i\colon Y_i\to X)_{i=1,\ldots,n}$  of $X$ over $K$  such that $ \Sigma \setminus \cup_{i=1}^n \pi_i(Y_i(K))$ is not dense in $X$.  
(The terminology strongly thin was first used in \cite{BSFPext}. )
   \item We say that $X$ has the \emph{weak Hilbert property over $K$} if $X(K)$ is not strongly thin.
   \end{enumerate}
\end{definition}

It seems reasonable to suspect that   abelian varieties over Hilbertian fields of characteristic zero satsify the weak Hilbert property, up to a  finite extension of the base field. A related question was first formulated in \cite[Question~1.44]{FehmJ}.

\begin{conjecture}\label{conj} Let $K$ be a Hilbertian field of characteristic zero.
Let $A$ be an abelian variety over $K$. Then there is a finite field extension $L/K$ such that $A_L$ has the weak Hilbert property.
\end{conjecture}
 
  Any field that is finitely generated over $\mathbb{Q}$ (e.g., a  number field) is Hilbertian.
By \cite{CDJLZ} the above conjecture is known when $K$ is a finitely generated field extension of $\mathbb{Q}$.

Any abelian extension of a Hilbertian field is Hilbertian.  In particular,  the maximal abelian extension $\mathbb{Q}^{\textrm{ab}}$ of $\mathbb{Q}$ is abelian. We note that Conjecture \ref{conj} is   known if $K$ is $\mathbb{Q}^{\textrm{ab}}$,  assuming the Frey--Jarden conjecture on ranks of abelian varieties; see \cite{BSFPabvar} for more precise statements.   

There are Hilbertian fields for which we do not know whether Conjecture \ref{conj} holds even in the case of elliptic curves.  

\begin{definition}
Let $k$ be a field of characteristic zero. A \emph{function field over $k$} is a finitely generated field extension $K$ of $k$ with transcendence  degree $\mathrm{trdeg}_k(K) \geq 1$.    
\end{definition}

If $k$ is a field  of characteristic zero and $K$ is a function field over $k$,  then  $K$ is Hilbertian \cite[Proposition~13.2.1]{FriedJarden}. Our main result verifies Conjecture \ref{conj} for  such fields.

\begin{alphtheorem}\label{thm:ab_var_has_WHP} 
Let $K$ be a function field over a field $k$ of characteristic zero and let  $A$ be an abelian variety over $K$. Then,  there is a finite field extension $L/K$ such that $A_L$ has the weak Hilbert property over $L$.
\end{alphtheorem}

 Theorem \ref{thm:ab_var_has_WHP} is a function field analogue of the fact that abelian varieties over number fields satisfy the potential weak Hilbert  property (as proven in \cite{CDJLZ}).  In fact, if the field $k$ in Theorem \ref{thm:ab_var_has_WHP} is itself  finitely generated over $\mathbb{Q}$, then Theorem \ref{thm:ab_var_has_WHP}  follows from the aforementioned main result of  \cite{CDJLZ}.  
Recall that, in \cite[Section~7]{CDJLZ},  the potential weak Hilbert property for abelian varieties over such a field $K$ is deduced from the potential weak Hilbert property for abelian varieties over number fields by a specialization argument.  However, Theorem \ref{thm:ab_var_has_WHP} is new if $k$ is, for example, algebraically closed. 

In fact, our proof of Theorem \ref{thm:ab_var_has_WHP} in case $K$ is finitely generated over $\mathbb{Q}$ with positive transcendence degree does not pass to the case of number fields via specialization.  We instead  appeal  to a sufficiently strong version of the weak Hilbert property for abelian varieties over complex function fields which relies on the work of  Xie--Yuan on the geometric Bombieri--Lang conjecture \cite{XieYuan1, XieYuan2}.  This strategy forces us to consider isotrivial (resp. non-isotrivial) abelian varieties separately:

\begin{definition} Let  $K/k$ be a field extension. A variety $X$ over $k$ is \emph{constant} (or \emph{$k$-constant}) or \emph{can be defined over $k$} if there is a variety $X_0$ over $k$ and an isomorphism $X \cong (X_0)_K$ over $K$.  A variety $X$ over $K$ is \emph{isotrivial} if $X_{\overline{K}}$ is $\overline{k}$-constant.
\end{definition}

\subsection{The case of elliptic curves}
The proof of Theorem \ref{thm:ab_var_has_WHP} for elliptic curves is considerably simpler, as one can  appeal to older results of De Franchis, Severi, and Grauert--Manin. We include a brief discussion of this proof before explaining the general strategy.

\begin{example} \label{example:elliptic_curves} Let $k$ be an algebraically closed field of characteristic zero and let $K$ be a  function field over $k$. Let $B$ be a variety over $k$ with function field  $K=K(B)$.
Let  $E$ be an elliptic curve over $K$.  To show that $E$ has the potential weak Hilbert property,  we consider two cases.
\begin{enumerate}
 \item  Assume that $E$ is non-isotrivial. Now, let $X\to E$ be a ramified cover with $X$ a smooth projective geometrically connected curve. Note that $X$ has genus at least two and that $X$ is non-isotrivial. In particular, it follows from the theorem of Grauert-Manin (or: the function field Mordell conjecture) that $X(K)$ is finite. In particular,  every infinite subset of $E(K)$ is not strongly thin.  Thus, it follows that $E(K)$ is not strongly thin if and only if $E(K)$ is dense.  If we now choose $L/K$ such that $E(L)$ is infinite (hence dense),  then $E_L$ has the weak Hilbert property over $L$.  
 \item Now,  assume that $E$ is isotrivial. Then, replacing $K$ by a finite field extension if necessary, there is an elliptic curve $E_0$ over $k$ and an isomorphism $E\cong (E_0)_K$ over $K$. Replacing $K$ by a finite field extension again if necessary, there is a point $P\in E(K)\setminus E_0(k)$ of infinite order. Now, let $X\to E$ be a ramified cover with $X$ a smooth projective geometrically connected curve (of genus at least two). If $X$ is non-constant, then $X(K)$ is finite (by Grauert--Manin again). If $X$ is constant, then $X=(X_0)_K$ for some curve $X_0$ over $k$. By the theorem of De Franchis, we have that $X(K)\setminus X_0(k)$ is finite.  If $X\to E$ is actually defined over $k$, then we are done: the set of points in $X_0(k)$ which map to $E(K)\setminus E_0(k)$ is empty.   However, if $X\to E$ is not defined over $k$,  a simple  rigidity argument shows that    the set of constant points $a$ in $X_0(k)$ which map to  the subgroup generated by $P$ is finite (see Section \ref{section:final}).  Thus, the set $E(K)$ is not strongly thin.  
\end{enumerate} 
This  concludes the proof of Theorem \ref{thm:ab_var_has_WHP} in the case of elliptic curves and   $k$  algebraically closed.
\end{example}

\subsection{Outline of proof of Theorem \ref{thm:ab_var_has_WHP}}

The proof of the potential weak Hilbert property for abelian varieties over number fields given in \cite{CDJLZ} relies on the geometry of abelian varieties over finite fields (and $p$-adic fields). It crucially   uses Chebotarev's density theorem and the Lang--Weil estimates.  As certain analogues of these theorems  do not hold over function fields such as $\mathbb{C}(t)$, to prove the potential weak Hilbert property for abelian varieties over function fields we     argue differently. Indeed, we instead appeal to known cases of the geometric Bombieri--Lang conjectures.

In fact, our proof of Theorem \ref{thm:ab_var_has_WHP} relies on the two theorems of Xie--Yuan concerning the geometric Bombieri--Lang conjecture \cite{XieYuan1, XieYuan2}.   These results essentially play the role of Grauert--Manin's theorem in Example \ref{example:elliptic_curves}.

 In the situation that the geometric trace of $A$ is zero, we obtain the strongest possible statement by combining Xie--Yuan's work on traceless abelian varieties \cite{XieYuan2} with Kawamata's structure theorem \cite[Theorem~23]{KawamataChar} and Yamanoi's work  (see Theorem \ref{thm:yamanoi}) on   the Green-Griffiths--Lang conjecture for varieties with maximal Albanese dimension.
 
\begin{alphtheorem}  \label{thm:nst_in_traceless_intro} Let $k$ be a field of characteristic zero and let $K$ be a  function field over $k$.
Let $A$ be an abelian variety over $K$ whose $\overline{K}/\overline{k}$-trace is zero.  If $\Omega\subset A(K)$ is a dense subset, then $\Omega$ is not strongly thin in $A$.
\end{alphtheorem}

An important difference between the number field and function field setting (in the presence of a ``trace'') is that the density of $A(K)$ is not sufficient for having the weak Hilbert property:

\begin{example}\label{example:constant_points} Let $A_0$ be an abelian variety over an algebraically closed field  $k$ of characteristic zero. Let $K = k(t)$.   Let $Y_0\to A_0$ be a surjective morphism and consider $Y= (Y_0)_{K}$ and $A = (A_0)_K$.   Then every point of $ A(K) = A_0(k)$ lifts to $Y(K) = Y_0(k)$. Thus,  $A(K)$ is dense, but $A$ does not have the weak Hilbert property over $K$.  
  \end{example}

Example \ref{example:constant_points} shows that the set of ``constant'' points on $A_K$ is strongly thin. A moment's thought reveals that, for every $P\in A(K)$, the set $P + A(k) $ of ``points that are constant up to translation by $P$'' is also strongly thin.  We will show that these   subsets of $A(K)$ are essentially the only  dense subsets which are strongly thin, assuming the field of constants $k$ to be algebraically closed and the abelian variety $A$ to be geometrically simple:

\begin{alphtheorem}\label{thm:constant_intro}   Let $k$ be an algebraically closed field of characteristic zero and let $K$ be a  function field over $k$.
   Let $A_0$ be a simple abelian variety over $k$ and let $A= (A_0)_K$. Let  $\Omega \subset A(K)$ be a dense subset such that, for every $P\in A(\overline{K})$, the subset  $\Omega  \cap  \left(A_0(k) + P\right)$ is not dense in $A$. Then $\Omega$ is not strongly thin in $A$.
\end{alphtheorem}

One  difficulty arising in the higher-dimensional case (which does not occur in Example \ref{example:elliptic_curves}) is that  an abelian variety might be non-isotrivial and have  non-trivial trace.  Since the geometric Bombieri--Lang conjecture is (currently)  not known for ramified covers of such abelian varieties,  to prove Theorem \ref{thm:ab_var_has_WHP} we will  patch the isotrivial and non-isotrivial case together using a suitable   product theorem  (see Theorem \ref{thm:product}). 

% It is at this point that we see the clear difference     between the geometric Bombieri--Lang conjecture for ramified covers of abelian varieties, and the potential weak Hilbert property for abelian varieties over function fields. The former is not known in general, but the latter is. 

\begin{ack}
We thank Arno Fehm for inspiring discussions on the Hilbert property and helpful comments on an early draft of this paper.
\end{ack}

\begin{con} If $k$ is a field, then a variety over $k$ is a geometrically integral finite type separated scheme over $k$.

Let $k$  be a field of characteristic zero, and let $k\subset K$ be a field extension such that $k$ is algebraically closed in $K$. If $A $ is an abelian variety, we let $\mathrm{Tr}_{K/k}(A)$ be its trace; see \cite[Definition~6.1 and Theorem~6.2]{ConradTrace}.  
\end{con}

\section{The weak Hilbert property over function fields}

The Chevalley--Weil theorem  \cite{CW} says that, given a finite \'etale cover of proper varieties over a number field $K$, we can find a finite extension $L/K$ such that all  $K$-points on the target lift to  $L$-points on the source.  This    is a consequence of Hermite's theorem on the smallness of the \'etale fundamental group of the ring of $S$-integers in a number field.  Replacing Hermite's theorem with the  fact that fundamental groups of complex algebraic varieties are finitely generated leads to the following function field analogue of Chevalley--Weil's theorem:  

\begin{lemma} [Geometric Chevalley--Weil]\label{lemma:cw} Let $k$ be an algebraically closed field of characteristic zero and let $K$ be a finitely generated field extension of $k$.
Let $f\colon Y\to X$ be a finite \'etale morphism of proper varieties over $K$. Then,  there is a finite field extension $L/K$ such that $X(K)$ is contained in $f(Y(L))$.
\end{lemma}
\begin{proof}
Let $B$ be a smooth variety with function field $K$.  We may choose a proper integral scheme $\mathcal{X}$ over $B$ with $\mathcal{X}_K=X$, a  proper integral scheme $\mathcal{Y}$ over $B$ with $\mathcal{Y}_K = Y$, and a  finite \'etale morphism $F\colon \mathcal{Y}\to \mathcal{X}$ with $F_K =f$.    Since $\mathcal{X}\to B$ is proper, for every $P\in X(K)$,  there is a dense open $U_P \subset B$ with complement of codimension at least two and a morphism $\sigma_P\colon U_P \to \mathcal{X}$  extending the morphism $P\colon \Spec K\to X$.  Consider the Cartesian diagram
\[\xymatrix{V_P \ar[rr]  \ar[d]_{\textrm{finite \'etale}}& & \mathcal{Y} \ar[d]^F \\ U_P \ar[rr]_{\sigma_P} & & \mathcal{X}} \]  
Note that the morphism $V_P\to U_P$ is finite \'etale surjective of degree $\deg(F) = \deg(f)$.  Since $B$ is a smooth variety and the complement of $U_P$ is of codimension at least two, by purity of the branch locus \cite[Theorem~X.3.1]{SGA1}, the finite \'etale morphism $V_P\to U_P$ extends to a finite \'etale cover of $B$ of degree $\deg(F)$.   Since $B$ is a smooth variety over an algebraically closed field of characteristic zero, the set of $B$-isomorphism classes of finite \'etale morphisms $B'\to B$ of bounded degree is finite (this follows from the fact that the \'etale fundamental group of $B$ is topologically finitely generated \cite[Th\'eor\`eme~2.3.1]{SGA7I}). In particular, there exists a finite \'etale morphism $\pi\colon C\to B$ such that, for every $P$ in $X(K)$ with associated   $\sigma_P\colon U_P\to \mathcal{X}$, the composed morphism $\pi^{-1}(U_P)\to U_P\to \mathcal{X}$ factors over $\mathcal{Y}\to \mathcal{X}$.  Defining $L$ to be the function field of $C$, this shows that   $X(K)$ is in the image of $Y(L)$, as required.
\end{proof}
 
Over function fields,  as in the number field setting, the Hilbert property forces the \'etale fundamental group to be trivial (as long as the base field $k$ is algebraically closed).
\begin{theorem} [Corvaja--Zannier, function field analogue] \label{thm:pi1} Let $k$ be an algebraically closed field of characteristic zero and let  $K$ be a function field over $k$. Let $X$ be a proper normal variety over $K$. If there is a finite field extension $L/K$ such that $X_L$ has the Hilbert property, then   $X_{\overline{K}}$ has no non-trivial finite \'etale covers.
\end{theorem}
\begin{proof}    Suppose for a contradiction that $X_{\overline{K}}$ has a non-trivial finite \'etale cover. Then, replacing $K$ by a finite field extension if necessary, we may assume that $X$ has the Hilbert property over $K$ and that there is a finite \'etale morphism $\pi\colon Y\to X$ of degree at least two with $Y$ a proper (geometrically integral) variety over $K$.  By Lemma \ref{lemma:cw}, there is a finite field extension $L/K$ such that $X(K)$ is contained in $\pi(Y(L))$.  It now follows from \cite[Proposition~1.5]{CZHP} (with ``$T=Y(L)$'' in \emph{loc. cit.}) that $X$ does not have the Hilbert property.   
\end{proof}

 \begin{remark}[Weak Hilbert property for curves over function fields] Let $k$ be an algebraically closed field of characteristic zero and let $K/k$ be a finitely generated field extension of with $\mathrm{trdeg}_k(K)\geq 1$. 
 Let $X$ be a smooth projective geometrically connected curve over    $K$.  One can show that $X$ has the Hilbert property over $K$ if and only if it is isomorphic to $\mathbb{P}^1_K$ (using, for example,  Theorem \ref{thm:pi1}).  
 Moreover,  if $X$ has positive genus, then  $X$ has the weak Hilbert property if and only if  either $X$ is a non-isotrivial genus one curve with $X(K)$ infinite or $X$ is a genus one curve defined over $k$, say $X = E_K$, and  $X(K)\setminus E(k)$ is infinite.
In particular, the curve $X$ satisfies the weak Hilbert property over some finite field extension if and only if it has genus at most one.  
\end{remark}

\begin{remark} Since nonzero abelian varieties have many non-trivial finite \'etale covers,  it follows that a nonzero abelian variety over $K$ (as in Theorem \ref{thm:pi1}) does not have the Hilbert property.  
\end{remark}

To prove the potential weak Hilbert property for abelian varieties over function fields we will first prove it for simple abelian varieties, and then use a product theorem. 

 In \cite[\S3.1]{Serre} Serre asked whether the product of varieties with the Hilbert property has the Hilbert property; this was proven in \cite{BSFP}. (Over number fields, this product property  can also be proven using   \cite[Lemma~8.12]{HarpazWittenberg}. ) The product theorem (for the Hilbert property) holds over arbitrary fields. Unfortunately, for the weak Hilbert property we do not have  such a general statement.  However, over finitely generated fields of characteristic zero,   the product of   smooth projective varieties with the weak Hilbert property has the weak Hilbert property \cite[Theorem~1.9]{CDJLZ}.  The reason for the restriction to finitely generated fields is related to the Chevalley--Weil theorem.  In fact, over fields where a ``Chevalley--Weil'' type property holds, the product theorem can be proven.   Following \cite{Luger2},  we place ourselves in the more general setting of varieties endowed with a non strongly thin subset, and have the following result.

\begin{theorem}\label{thm:product}  Let $K$ be a field of characteristic zero. Assume that either of the following holds:
\begin{enumerate}
\item There is an   algebraically closed field $k$  contained in $K$ such that $K $ is a function field over $k$.
\item The field $K$ is  finitely generated   over $\mathbb{Q}$.
\end{enumerate} Let $X$ and $Y$ be normal projective varieties over $K$. Let $\Sigma \subset  (X\times Y)(K)$ be a subset, and define $\Sigma_X = p(\Sigma)$, where $p\colon X\times Y\to X$ is the projection onto the first factor. Suppose that $\Sigma_X $ is not strongly thin in $X$ and that, for every $x$ in $\Sigma_X$, the subset $\Sigma \cap \left( \{x\} \times Y\right)$ is not strongly thin in $Y$. Then $\Sigma$ is not strongly thin in $X\times Y$.
\end{theorem}
 \begin{proof}
 If $K$ is finitely generated over $\mathbb{Q}$, then this is a consequence of  \cite[Theorem~4.4]{Luger2} (which generalizes \cite[Theorem~1.9]{CDJLZ}). If $K$ is finitely generated over an algebraically closed field $k$ of characteristic zero, the proof  in \emph{loc. cit.} can be easily  adapted to proving the desired statement, simply by replacing the usual Chevalley--Weil theorem  with the geometric Chevalley--Weil theorem (Lemma \ref{lemma:cw}) whenever necessary.
 \end{proof}

\section{Traceless abelian varieties} 
%As explained in the introduction of \cite{CDJLZ} and \cite[Remark~?]{FJ}, Lang's conjecture on the non-density of rational points for varieties of general type over a number field implies the weak Hilbert property for abelian varieties.  This line of reasoning can be adapted to the function field setting, as long as the abelian variety is traceless.   This then culminates to Corollary \ref{cor:nst_in_traceless} below. \ari{ADD reference to LAng}

To prove the main result of this section (Theorem \ref{thm:nst_in_traceless_intro}),  we need the following definitions:

\begin{definition}
A proper variety $X$ over a field $k$ of characteristic zero is of \emph{general type} if it has a desingularization $Y\to X$ with $\omega_{Y/k}$ big.
\end{definition}

\begin{definition}
Let $X$ be a proper variety over a field $k$ of characteristic zero. 
The \emph{special locus} $\mathrm{Sp}(X)$ of $X$ is the closure of  $\cup_{A, f} \mathrm{Im}(f)$ in $X_{\overline{k}}$, where the union runs over all abelian varieties $A$ over $\overline{k}$ and all non-constant rational maps $A\dashrightarrow X_{\overline{k}}$. 
\end{definition}

We can now state  the finiteness  theorem of  Xie--Yuan   on ramified covers of traceless abelian varieties \cite{XieYuan2}.  

\begin{theorem}[Xie--Yuan, traceless part]\label{thm:xy_traceless} Let $k$ be a field of characteristic zero and  let $k\subset K$ be a finitely generated field extension with $\mathrm{trdeg}_k(K)\geq 1$. Let $A$ be an abelian variety over $K$.  Suppose that the $\overline{K}/\overline{k}$-trace of $A_{\overline{K}}$ is zero.
Let $Z\to A$ be a  ramified cover. Then $Z(K)\setminus \mathrm{Sp}(Z)$ is finite. 
\end{theorem}

This finiteness result is ``non-trivial'' precisely when $\mathrm{Sp}(Z)\neq Z$.   We note that Yamanoi proved that the latter holds if and only if $Z$ is of general type  \cite{Yamanoi1}:

\begin{theorem}[Yamanoi]\label{thm:yamanoi} Let $A$ be an abelian variety over a field $K$ of characteristic zero.
 Let $Z\to A$ be a finite morphism with $Z$  a normal projective variety. Then $Z$ is of general type if and only if $\mathrm{Sp}(Z) \neq Z$.
\end{theorem}

\begin{remark}\label{remark:prod_of_ell_curves} Ramified covers of elliptic curves are always of general type.  More precisely,  let $E$ be an elliptic curve, and let $C\to E$ be a ramified cover with $C$ a smooth projective curve.  Note that $C$ is then of general type. However, ramified covers of higher-dimensional abelian varieties are not necessarily of general type.  For example,  if $Z = C \times E$, then $Z\to E \times E$ is a ramified cover of $A =E\times E$. Note that in this case $Z$ is covered by elliptic curves (so $\mathrm{Sp}(Z) =Z$) and that $Z$ is not of general type (as it has Kodaira dimension one).
\end{remark}

We see in Remark \ref{remark:prod_of_ell_curves} that an abelian variety may have ramified covers which are not of general type.  We can not directly apply Xie--Yuan's theorem to such covers to deduce non-density of rational points.  However, as is the case in the specific example of Remark \ref{remark:prod_of_ell_curves},  every ramified cover  dominates (up to finite \'etale cover) another ramified cover \emph{of general type} of some  quotient of $A$   (this was proven by Kawamata  \cite[Theorem~23]{KawamataChar}). This leads to the following:

\begin{corollary} Let $k$ be a field of characteristic zero and let $K$ be a finitely generated field extension of $k$.
  Let $A$ be an abelian variety over $K$.  Suppose that the $\overline{K}/\overline{k}$-trace of $A$ is zero.
Let $Z\to A$ be a  ramified cover.  Then $Z(K) $ is not dense.
\end{corollary}
\begin{proof}  We may assume that $k$ is algebraically closed and replace $K$ by a finite field extension if necessary.  Then,  
 by Kawamata's structure theorem \cite[Theorem~23]{KawamataChar}, there is an abelian variety $B$,  a surjective homomorphism $A\to B$, a finite \'etale morphism $Z'\to Z$, a surjective morphism $Z'\to X$ and a ramified cover $X\to B$,  where $X$ is a positive-dimensional normal projective variety of general type:
\[
\xymatrix{ Z' \ar[d]_{\textrm{surjective}} \ar[rr]^{\textrm{finite \'etale}}  & & Z \ar[r]    & A \ar[d] \\
 X \ar[rrr] & & & B &  }
\]
Since $X$ is a ramified cover of an abelian variety and $X$ is of general type,  it follows from Yamanoi's theorem (Theorem \ref{thm:yamanoi})  that $\mathrm{Sp}(X) \neq X$.  Also, since $A$ has geometric trace zero, we have that $B$ has geometric trace zero. In particular, by Xie--Yuan's theorem (Theorem \ref{thm:xy_traceless}),  we have that $X(L)$ is not dense for any finite field extension $L/K$.  It follows directly that $Z'(L)$ is not dense for any finite field extension $L/K$.  Then, by (geometric) Chevalley--Weil (Lemma \ref{lemma:cw}), we conclude that $Z(K)$ is not dense, as required. 
\end{proof}

As a direct consequence, we obtain the strongest possible conclusion pertaining to strongly thin subsets of abelian varieties, thereby proving   Theorem \ref{thm:nst_in_traceless_intro}.

%\begin{corollary}\label{cor:nst_in_traceless}  Let $k$ be a field of characteristic zero and let $K$ be a function field over $k$.  Let $A$  be an abelian variety over $K$ whose  $\overline{K}/\overline{k}$-trace is zero.
%If $\Omega\subset A(K) $ is a dense subset, then $\Omega$ is not strongly thin in  $A$.
%\end{corollary}   

  \section{Constant abelian varieties}\label{section:final}
The Bombieri--Lang conjecture for ramified covers of a  constant abelian variety is currently not known, unless we restrict our attention to ``hyperbolic'' ramified covers \cite{XieYuan1}.

\begin{theorem}[Xie--Yuan,  hyperbolic covers] \label{thm:xy_groupless} Let $k$ be a field of characteristic zero and let $K$ be a finitely generated field extension of $k$. Let $A$ be an abelian variety over $K$.  
Let $Z\to A$ be a  ramified cover over $K$.  Suppose that $\mathrm{Sp}(Z) = \emptyset  $ and that   $Z(K)$ is dense.  
Then  $A$ and $Z$ can be defined over $k$.  Moreover, if $Z_0$ is  a variety over $k$ with  $Z = (Z_0)_K$, then $Z(K)\setminus Z_0(k)$ is not dense.
\end{theorem}
 
 %We stress that the density of rational points on a groupless ramified cover of an abelian variety over $K$ implies that the cover and the abelian variety can be defined over $K$. This is  why below we restrict our attention to constant abelian varieties primarily.

\begin{corollary}\label{cor:xy_groupless_covers} Let $k$ be an algebraically closed field of characteristic zero and let $K $ be a function field over $k$.
   Let $A_0$ be an abelian variety over $k$ and let $A= (A_0)_K$. 
   Let  $\Omega \subset A(K)$ be a dense subset such that, for every $P$ in $A(\overline{K})$, the subset  $ \Omega \cap  \left( A_0(k) + P\right)$ is not dense in $A$.   Let  $(\pi_i\colon Z_i\to A)_{i=1,\ldots,n}$ be a finite collection of ramified covers over $K$ with each $Z_i$ a   normal projective variety. Suppose that for every $i$, we have $\mathrm{Sp}(Z_i) = \emptyset$. Then    
   \[
   \Omega \setminus \cup_i \pi_i(Z_i(K))
   \] is dense.
   \end{corollary}
   \begin{proof}    We may assume that there is an integer $1\leq m\leq n$ such that $Z_i(K)$ is dense if and only if $1\leq i\leq m$.  For such $i$,   it follows from Theorem \ref{thm:xy_groupless} that  $Z_i$ can be defined over $k$, i.e.,  there is a normal projective variety $Z_{i,0}$ over $k$ and an isomorphism $Z_i \cong (Z_{i,0})_K$  over $K$.     Thus,  by   Lemma \ref{lemma:deformation_theory} below,  there is a point $P_i\in A(K)$ 
 such that $f_i := \tau_{P_i} \circ \pi_i$ is defined over $k$.   In particular, the subset  $ \pi_i(Z_{i,0}(k)) \subset  A(K)$ is contained in $A_0(k) - P_i$.     Thus,  the intersection $\Omega \cap \pi_i(Z_i(K))$ is contained in the union of  $\Omega \cap \pi_i(Z_{i,0}(k))$ and 
 $ \pi_i(Z_i(K))\setminus \pi_i(Z_{i,0}(k))$. The former is not dense by assumption (as it is contained in $\Omega \cap\left( A_0(k) -P_i\right)$) and the latter is not dense by the final statement  of  Theorem \ref{thm:xy_groupless} (which says that $Z_i(K)\setminus Z_{i,0}(k)$ is not dense).  
   \end{proof}
   
    If $A$ is an abelian variety over a field $k$ and $P\in A(k)$ is a point, then $\tau_P\colon A\to A$ denotes the translation by $P$ on $A$.

\begin{lemma}\label{lemma:deformation_theory} Let $k$ be an algebraically closed field of characteristic zero and let  $K$ be a function field over $k$. Let $A$ be an abelian variety over $k$ and let $Z$ be a normal projective variety over $k$.
Let $f\colon Z_K\to A_K$ be a finite ramified morphism  over $K$. Then, there is a point $P\in A(K)$  such that $\tau_P\circ f$ is  defined over $k$.
\end{lemma}
   \begin{proof} This follows from the rigidity of homomorphisms between abelian varieties. More precisely,  let $z\in Z(k)$ be a point, and let $a\colon Z\to \mathrm{Alb}_Z$ be the Albanese map of $Z$ over $k$ with $a(z) =0$.   Let $a_K\colon Z_K\to (\mathrm{Alb}_Z)_K$ be the Albanese map of $Z_K$, let $Q=f(z) \in A(K)$ and write $P=-Q$.  By the universal property of Albanese varieties,  there is a (unique) homomorphism $h\colon (\mathrm{Alb}_Z)_K\to A_K$ such that $  f =  \tau_Q\circ  h\circ a_K$. Since $h$ is a homomorphism between abelian varieties and $k$ is algebraically closed, the homomorphism $h$  can be defined over $k$. This   implies that $h\circ a_K$ and thus $\tau_P\circ f$, can be defined over $k$.
    %Let $g=\dim A$. Let $\underline{\Hom}_k(Z,A)$ be the moduli scheme of morphisms from $Z\to A$. Let $H$ be the connected component of $\underline{\Hom}_k(Z,A)$ containing $f$, and note that $H$ is a quasi-projective scheme (by the theory of Hilbert schemes).  Note that $A$ acts freely on $H$ via translation.  In particular,  we have that $\dim H\geq g$. Since $T_A$ is trivial, we have that $\mathrm{H}^0(Z,f^\ast T_A) $ is $g$-dimensional.  In particular, the tangent space  of $H$ at $f$ is $g$-dimensional \cite[\S 2.3]{Debarrebook}.  We conclude that $H$ is smooth of dimension $g$.  We infer that the action of $A$ on $H$ is transitive, and thus $H$ is an $A$-torsor over $k$.  Since $f$ is a $K$-point of the $A$-torsor $H$,   the natural morphism $A_K\to  H_K$ given by $a\mapsto \tau_a\circ f$ is an isomorphism. This concludes the proof. (Let $g\in H(k)\subset H(K)$ and choose $P\in A(K)$ such that $P$ is mapped to $g$ via $A_K\to H_K$. Then $\tau_P\circ f = g$ and thus $\tau_P\circ f$ is defined over $k$.)
   \end{proof}

   We now prove the characterization of strongly thin dense subsets in constant simple abelian varieties stated in the introduction (see  Theorem \ref{thm:constant_intro}).
   
   \begin{corollary}\label{cor:constant1} Let $k$ be an algebraically closed field of characteristic zero and let $K$ be a function field over $k$.
   Let $A_0$ be a simple abelian variety over $k$ and let $A= (A_0)_K$. Let  $\Omega \subset A(K)$ be a subset such that, for every $P\in A(\overline{K})$, the subset  $\Omega  \cap  \left(A_0(k) + P\right)$ is not dense in $A$. Then $\Omega$ is not strongly thin.
   \end{corollary}
   \begin{proof}
   This follows from Corollary \ref{cor:xy_groupless_covers} and the fact that, for every ramified cover $Z\to A$ of a simple abelian variety $A$, we have that $\mathrm{Sp}(Z) = \emptyset$.   (The latter can be proven as follows.  Suppose that $Z\to A$ is a ramified cover of a simple abelian variety with $\mathrm{Sp}(Z)\neq \emptyset$.  Then,  replacing $K$ by a finite field extension if necessary, there is an abelian variety $B$  over $K$ and a non-constant morphism $B\to Z$. The image of the composed morphism $B\to Z\to A$ is a non-trivial abelian subvariety of $A$, hence equal to $A$. In particular, $B\to A$ is surjective and thus $B\to Z$ is surjective.  Since $B\to A$ is smooth,  this implies that $Z\to A$ is \'etale, contradicting our assumption that $Z\to A$ is ramified.)
   \end{proof}

 \begin{lemma}\label{lemma:nondeg_int} Let $k$ be an algebraically closed field of characteristic zero and let $K$ be a function field over $k$.
 Let $A_0$ be an abelian variety over $k$, and let $a\in A(K)\setminus A_0(k)$ be a non-degenerate point.  Let $\Omega$ be the subgroup generated by $a$. Then, for $P\in A(\overline{K})$, the set $\Omega \cap (A_0(k) +P)$ has at most one element.
 \end{lemma}
 \begin{proof} Suppose that there are integers $m>n$ with  $n a\in A_0(k) + P$ and $ma \in A_0(k) + P$. Then $(m-n)a$ is in $A_0(k)$. Since $A_0(k)$ is divisible, it follows that $a\in A_0(k)$ contradicting our assumption that $a\in A(K)\setminus A_0(k)$.  
 \end{proof}
 
A point $P$ on an abelian variety $A$ is \emph{non-degenerate} if the group $\mathbb{Z}\cdot P$ generated by $P$ is dense in $A$.   It is well-known that an abelian variety $A$ over a field of characteristic zero $k$ contains many non-degenerate points over the algebraic closure $\overline{k}$. In the case of function fields, this is not directly useful, as we will need the existence of non-degenerate points which do not come from the field of constants. This can be deduced from the work of Frey--Jarden \cite{FreyJarden},  but also  by adapting an argument of Mazur given in the proof of \cite[Theorem~3.1]{HassettTschinkel}.

 \begin{lemma}[Existence of non-degenerate points]\label{lemma:existence_nondegpoint}
 Let $k$ be a field of characteristic zero and let $K$ be a function field over $k$. If $A$ is an abelian variety over $k$, then there is a finite field extension $L/K$  and a non-degenerate point $P\in A(L)\setminus A(k)$. 
 \end{lemma}

We are now ready to prove the potential weak Hilbert property for abelian varieties over a function field:

\begin{proof}[Proof of Theorem \ref{thm:ab_var_has_WHP}]
 We may replace $K$ by a finite field extension and write $A =\prod_i A_i$ with each $A_i$ a geometrically simple abelian variety, as the weak Hilbert property descends along finite \'etale morphisms \cite[Proposition~3.4]{CDJLZ}.
We may  also assume that $k$ is algebraically closed in $K$. Let $B$ be a smooth (geometrically connected) variety over $k$ with function field $K$, and let $K'$ be the function field of $B_{\overline{k}}$.  Replacing $K$ by a finite field extension if necessary, we may assume that,  for every $i$, we have that either the $\overline{K}/\overline{k}$-trace of $A_i$ is zero or that $A_i$ is constant.
 
 If $A_i$ has $\overline{K}/\overline{k}$-trace zero,  replacing $K$ by a finite field extension if necessary, we may  choose a non-degenerate point $a_i \in A_i(K)$. Define $\Omega_i \subset A_i(K) $ to be the subgroup generated by $a_i$. Note that $\Omega_i\subset A_i(K)\subset A_i(K')$ is not strongly thin in $A_{i,K'}$ by Theorem \ref{thm:nst_in_traceless_intro}. (Note that this conclusion is a priori stronger than being non strongly thin in $A_i$.)

If  $A_i$ is constant,  there is  an abelian variety $A_{i,0}$ over $k$ such that $A_i \cong (A_{i,0})_K$.   
By applying   Lemma \ref{lemma:existence_nondegpoint} to $(A_{i,0})_{\overline{k}}$, we see that  there is a finite field extension $L'/K'$ and a non-degenerate point $a_i\in A_{i,0}(L')\setminus A_{i,0}(\overline{k})$.   Since $K'/K$ is algebraic,  replacing $K$ by a finite field extension if necessary, we may assume that $a_i \in A_{i,0}(K) = A_i(K)$.     Let $\Omega_i  \subset A_i(K)$ be the subgroup generated by $a_i$.  Note that $\Omega_i$ is not strongly thin by Corollary \ref{cor:constant1} and Lemma \ref{lemma:nondeg_int}.

Define $\Sigma \subset A(K)$ to be $\prod_i \Sigma_i$.   By the product theorem (Theorem \ref{thm:product}),  the subset $\Sigma$ is not strongly thin in $A_{K'}$. In particular, it is not strongly thin in $A$. This shows that $A$ has the weak Hilbert property, as required.
\end{proof}

\bibliography{lafon}{}
\bibliographystyle{alpha}

\end{document}